\documentclass[]{amsart}
\usepackage{amscd,amssymb,amsmath,verbatim,hyperref,
epsfig,latexsym,subfigure, supertabular, graphicx}
\usepackage[arrow,matrix,graph,frame,poly,arc,tips]{xy}
\usepackage{color}
\usepackage{tikz}
 \usepackage{diagbox}
\DeclareGraphicsExtensions{.pdf,.jpg, .png}
\newtheorem{thm}{Theorem}[section]

\newtheorem{lemma}[thm]{Lemma}
\newtheorem{exm}[thm]{Example}
\newtheorem{prop}[thm]{Proposition}

\theoremstyle{definition}
\newtheorem{defin}[thm]{Definition}
\newtheorem{example}[thm]{Example}
\newtheorem{remark}[thm]{Remark}
\newtheorem{question}[thm]{Question}
\newtheorem{problem}[thm]{Problem}

\numberwithin{equation}{section}

\newcounter{first}
{\end{list}}

\newcommand{\A}{{\mathcal A}}

\newcommand{\Z}{{\mathbb Z}}

\newcommand{\C}{{\mathbb C}}
\renewcommand{\P}{{\mathbb P}}
\renewcommand{\k}{\Bbbk}
\newcommand{\kk}{\mathbb K}

\DeclareMathOperator{\codim}{codim}
\DeclareMathOperator{\spn}{span}
\DeclareMathOperator{\pdim}{pdim}

\DeclareMathOperator{\reg}{reg}

\begin{document}

\title[Nets in $\P^2$ and Alexander Duality]%
{Nets in $\P^2$ and Alexander Duality}

\author[Nancy Abdallah]{Nancy Abdallah}
\address{Department of Mathematics,
University of Boras,
Boras, Sweden}
\email{\href{mailto:nancy.abdallah@hb.se}{nancy.abdallah@hb.se}}
\urladdr{\href{https://www.hb.se/en/shortcuts/contact/employee/NAAB}%
{https://www.hb.se/en/shortcuts/contact/employee/NAAB}}

\author[Hal Schenck]{Hal Schenck}
\thanks{Schenck supported by NSF 2006410}
\address{Department of Mathematics,
Auburn University, Auburn, AL 36849}
\email{\href{mailto:hks0015@auburn.edu}{hks0015@auburn.edu}}
\urladdr{\href{http://webhome.auburn.edu/~hks0015/}%
{http://webhome.auburn.edu/~hks0015/}}


\date{\today}

\begin{abstract}
A  net in $\P^2$ is a configuration of lines $\A$ and points $X$ satisfying certain incidence properties. 
Nets appear in a variety of settings, ranging from quasigroups to combinatorial design to classification of Kac-Moody algebras to cohomology jump loci of hyperplane arrangements.
For a matroid $M$ and rank $r$, we associate a monomial ideal (a monomial variant of the Orlik-Solomon ideal) to the set of flats of $M$ of rank $\le r$. 
In the context of line arrangements in $\P^2$, applying Alexander duality to the resulting ideal yields insight into the combinatorial structure of nets. 
\vskip -5in
\end{abstract}
\vskip -5in
\maketitle
\vskip -5in
\section{Introduction}\label{sec:intro}
The investigation of point-line incidence relations in $\P^2$ reaches back into the mists of time; for a comprehensive treatment see Gr\"unbaum \cite{Gr}.
\begin{defin}\label{NetDef}
For a configuration of lines $\A \subset \P_{\C}^2$, if $p$ is an intersection point of two or more lines, define $\mu(p) = |\mbox{lines through }p|-1$, and let $L_2(\A)$ be the set of all intersection points.
A $(k,d)$ net is a partition $\Pi$ of the lines of $\A$ into $k\ge 3$ blocks, each containing $d\ge 2$ lines, and a subset $X \subseteq \{p \in L_2(\A) \mid \mu(p) \ge 2 \}$ of multiple points such that 
\begin{enumerate}
\item every pair of lines from distinct blocks meet in some $ p \in X$.
\item there is exactly one line from each block of $\Pi$ passing through a point $p\in X$.
\end{enumerate} 
A {\em potential} $(k,d)$ net is a partition $\Pi$ and subset $X$ as above, but without the requirement that conditions $(1)$ and $(2)$ hold. 
\end{defin}
\noindent If $(\Pi, X)$ is a $(k,d)$ net, then it is easy to show that every line meets $X$ in $d$ points, and $|X|=d^2$. In \cite{Y09} Yuzvinsky proves that a $(k,d)$ net must have $k \in \{3,4\}$.
\begin{exm}\label{braidarr}
For $1\le i < j \le 4$ the equations $x_i-x_j =0$ define a set of hyperplanes in $\C^{4}$ which all contain the subspace $W = \spn(1,\cdots,1)$.
Projecting to $W^\perp$ yields a configuration of planes $A_3 \subseteq \C^{3}$ with common intersection at the origin, so $A_3$ also defines a line configuration in $\P^2$.
\vskip .15in
\begin{figure}[ht]
\subfigure{%
\label{fig:braid4}%
\begin{minipage}[t]{0.35\textwidth}
\setlength{\unitlength}{16pt}
\begin{picture}(6,5)(-1,-2.5)
\put(-0.8,0){\line(1,0){5.6}}
\put(-0.8,-0.4){\line(2,1){4.7}}
\put(-0.4,-0.8){\line(1,2){2.8}}
\put(2,-0.8){\line(0,1){5.6}}
\put(4.8,-0.4){\line(-2,1){4.7}}
\put(4.4,-0.8){\line(-1,2){2.8}}
\put(-1.7,0.3){\makebox(0,0){$L_{1}$}}
\put(-1.5,-0.6){\makebox(0,0){$L_{2}$}}
\put(-0.5,-1.3){\makebox(0,0){$L_{3}$}}
\put(2,-1.3){\makebox(0,0){$L_{4}$}}
\put(4.5,-1.3){\makebox(0,0){$L_{5}$}}
\put(5.5,-0.6){\makebox(0,0){$L_{6}$}}
\end{picture}
\end{minipage}
}
\setlength{\unitlength}{0.8cm}
\subfigure{%
\label{fig:partitionlat}%
\begin{minipage}[t]{0.45\textwidth}
\begin{picture}(5,5.7)(0,-4)
\xygraph{!{0;<10mm,0mm>:<0mm,14mm>::}
[]*D(3){{\bf 0} = \hat{1}}*-{\bullet}  
(
-@{..}[dlll]*D(3.4){123}*-{\bullet}  
(
-@{..}[d]*U(2.5){1}*-{\bullet}-@{..}[drrr]*U(2.5){\C^3 =\hat{0} }*-{\bullet} 
,-@{..}[dr]*U(2.5){2}*-{\bullet}-@{..}[drr]
,-@{..}[drr]*U(2.5){3}*-{\bullet}-@{..}[dr]
)
,-@{..}[dll]*D(2.1){25}*-{\bullet}  
(
-@{..}[d]
,-@{..}[drrrr]*U(2.5){5}*-{\bullet}-@{..}[dll]
)
,-@{..}[dl]*D(3.4){156}*-{\bullet} 
(
-@{..}[dll]
,-@{..}[drr]*U(2.5){6}*-{\bullet}-@{..}[dl]
,-@{..}[drrr]
)
,-@{..}[d]*D(2.1){36}*-{\bullet} 
(
-@{..}[dl]
,-@{..}[dr]
)
,-@{..}[dr]*D(3.4){246}*-{\bullet}  
(
-@{..}[dlll]
,-@{..}[d]
,-@{..}[drr]*U(2.5){4}*-{\bullet}-@{..}[dlll]
)
,-@{..}[drr]*D(2.1){14}*-{\bullet}  
(
-@{..}[dlllll]
,-@{..}[dr]
)
,-@{..}[drrr]*D(3.4){345}*-{\bullet}  
(
-@{..}[dllll]
,-@{..}[dl]
,-@{..}[d]
)
(
}
\end{picture}
\end{minipage}
}
\caption{\textsf{The line configuration $A_{3}\subseteq \P^2$ and
    $L(A_{3})\mbox{ for }A_{3} \subseteq \C^3$}}
\label{fig:braid}
\end{figure}
\vskip -.05in
\noindent  The matroid defined by the lattice of intersections $($in $\C^3)$ is depicted on the right. The partition $|14|25|36|$ and set of triple points $X = \{123,156,246,345\}$ define a net.

\end{exm}
\noindent It follows immediately from Definition~\ref{NetDef} that if $(\Pi,X)$ is a net, then every $\cap_i H_i =p  \in L_2(\A)$ is either an element of $X,$ or has all $H_i$ in the same block of $\Pi$. If $H_1 \cap H_2 = p\in L_2(\A)$ and $\mu(p)=1$, then $H_1$ and $H_2$ must be in the same block of $\Pi$. \newline

\noindent The set of flats of a matroid, partially ordered by inclusion, form a lattice, so it is natural to ask:
\begin{question}\label{monIquestion}
Is there a monomial ideal associated to a matroid that captures existence of a net?
\end{question}
\begin{defin}\label{JidealDef}
For a matroid on ground set $\{1,\ldots,n\}$ and choice of rank $r$ and field $\k$, let $S = \k[x_1, \ldots, x_n]$, and let $J$ denote the ideal generated by monomials corresponding to the flats of rank $\le r$. So a monomial $m = x_{i_1}\cdots x_{i_k} \in J \leftrightarrow [i_1,\ldots,i_k]$ is a flat of rank at most $r$. 
\end{defin} 
\begin{exm}\label{Jbraidarr}
For $r=2$, the ideal $J$ in Example~\ref{braidarr} is generated by 
\[
\langle x_1x_4, x_2x_5, x_3x_6,x_1x_2x_3, x_1x_5x_6, x_2x_4x_6, x_3x_4x_5 \rangle.
\]
\end{exm}
Definition~\ref{JidealDef} works for any matroid; our interest stems from the study of complex projective hyperplane arrangements. In that setting, a flat of rank $r$ corresponds to a (maximal) collection of hyperplanes meeting in codimension $r$. Work of Falk-Yuzvinsky in \cite{FY} shows that nets play a fundamental role in the study of the resonance variety of a hyperplane arrangement. 

The resonance variety is defined in terms of the Orlik-Solomon algebra, and has attracted considerable attention: see for example work of Aomoto \cite{Ao}, Esnault-Schechtman-Viehweg \cite{ESV}, Schechtman-Terao-Varchenko \cite{STV}, Yuzvinsky \cite{Yuz}, Falk \cite{Fa},  Cohen-Suciu \cite{CScv}, Libgober-Yuzvinsky \cite{LY}, and Falk-Yuzvinsky \cite{FY}. The Orlik-Solomon algebra is not needed to describe nets, but is used to define resonance varieties. For completeness we include in \S 5 an appendix on the Orlik-Solomon algebra and resonance varieties.

The generalization of Example~\ref{braidarr} will serve as a running example. The {\em braid arrangement} $A_n$ is defined by equations $x_i-x_j = 0$ for $1 \le i <j \le n+1$. It plays a central role in many areas of mathematics: in mathematical physics, the complement of $A_n$ is the configuration space for $n+1$ non-colliding points. In combinatorics, the lattice of intersections $L(A_n)$ is isomorphic to the partition lattice $\Pi_{n+1}$, and in representation theory, $A_n$ consists of fixed points of reflections in the Weyl group of $SL(n+1)$. 

\begin{remark}\label{EvsS}
By Proposition 2.1 of \cite{AAH}, for a squarefree monomial ideal, results over the symmetric algebra may be translated to results over the exterior algebra, and vice versa. In this paper, we work over the symmetric algebra.
\end{remark}
For hyperplane arrangements, a natural first guess at answering Question~\ref{monIquestion} is the initial ideal of the Orlik-Solomon ideal, which has
been used to good effect in a number of settings, e.g. Bj\"orner-Ziegler \cite{BZ}. It turns out that the initial ideal loses too much combinatorial
information to be useful in identifying nets and resonance. The ideal $J$ appearing in Definition~\ref{JidealDef} is our proposed answer to Question~\ref{monIquestion}. As our main interest is in nets, we focus on the rank two case, and in this setting call the ideal $J$ the {\em monomial OS ideal}.

A main tool in our investigation is Alexander duality. Alexander duality is a staple of both algebraic topology, commutative algebra, and combinatorics. A combinatorial proof of Alexander duality appears in \cite{BT}. In commutative algebra Alexander duality plays a key role in the study of squarefree monomial ideals. Applications of Alexander duality related to matroids and arrangements appear in works of Bj\"orner-Ziegler \cite{BZ92}, Falk \cite{f92}, and Eisenbud-Popescu-Yuzvinsky \cite{EPY}, but none address Question~\ref{monIquestion}.

In \S 2 we review some necessary concepts from homological and commutative algebra: free resolutions, Castelnuovo-Mumford regularity, betti numbers, and algebraic Alexander duality. In \S 3 we use Alexander duality to make a connection between nets and the ideal $J$, and in \S 4 we use Alexander duality to study the braid arrangement. As noted above, the appendix of \S 5 is a quick primer on the Orlik-Solomon algebra and resonance varieties. \newline

\pagebreak
\subsection{Approach and main results}  We state all results in terms of the symmetric algebra, which by Remark~\ref{EvsS} can can be translated to the exterior setting if desired.
\vskip .05in
\noindent {\bf Theorem A}:  
If $\A=(\Pi, X)$ is a potential $(k,d)$ net, with $\Pi = (\pi_1,\ldots,\pi_k)$, let $J_\Pi$ denote the ideal generated by the $k$ monomials of degree $d$ given by products of the variables within each block $\pi_i$,  and let $J_X$ denote the ideal generated by monomials corresponding to the points of $X$. Hence $J = J_X+J_Y$, where $J_{Y}$ is generated by monomials corresponding to elements of $L_2(\A)$ not in $X$. As noted earlier the degree two component of $J$ satisfies $J_2 \subseteq J_Y$. 
  \begin{enumerate}
    \item $\A=(\Pi, X)$ is a $(k,d)$ net iff 
    \[
    (J_X^\vee)_d = (J_{\Pi})_d, 
    \]
    where $I^\vee$ denotes the Alexander dual of $I$, defined in \S 2. 
    \vskip .05in
    \item If $\A = (\Pi,X)$ is a $(k,d)$ net, then all intersections of lines within a block $\pi_i$ of $\Pi$ are normal crossing iff $J_2$ is a direct sum of $k$ determinantal ideals $J_{\pi_1},\ldots,J_{\pi_k}$, with each $J_{\pi_i}$ generated by all squarefree quadratic monomials in a block of $d$-variables, and the blocks are disjoint.
                \end{enumerate}
   \vskip .03in  
 In the setting of $(2)$, a corollary is that each $J_{\pi_i}$ is the ideal of the $2 \times 2$ minors of a $2 \times d$ matrix, and the 
  quadratic quotient $S/J_2$ factors as 
  \vskip -.13in
  \[
   S/J_2 \simeq \bigotimes_{i=1}^k S_{\pi_i}/J_{\pi_i}, \mbox{ with } S_{\pi_i} \mbox{ a polynomial ring in the variables of }\pi_i.
   \]
This in turn means that the free resolution of $S/J_2$ as an $S$-module is the tensor product of the free resolutions of the $S/J_{\pi_i}$. Hence the free resolution of $S/J_2$ is a tensor product of $k$ Eagon-Northcott complexes, so $S/J_2$ is arithmetically Cohen-Macaulay (which means that $\codim(J_2)=\pdim(S/J_2)$), and the Alexander dual $S/J_2^\vee$ has a linear free resolution (which means that all differentials appearing in the free resolution of $S/J^\vee$ after the first step are matrices of linear forms).
             \vskip .03in
      \noindent Theorem B is an analysis of the monomial OS algebra for the type $A$ reflection arrangements generalizing Example~\ref{braidarr}. An arrangement of type $A_n$ is defined by the vanishing of the linear forms $x_i-x_j$ for $1 \le i < j \le n+1$, which is also the graphic arrangement (see \cite{ot}, \S2.4) corresponding to the complete graph $K_{n+1}$. The monomial OS ideal for $K_{n+1}$ is generated by cubic monomials corresponding to triangles in the graph, and quadratic monomials corresponding to pairs of disjoint edges.
      \vskip .03in
\noindent  {\bf Theorem B}: 
For the braid arrangement $A_{n-1}$, let $J(K_n)$ denote the ideal $J$. Then
the Hilbert series is $P(S/J(K_n),t)/(1-t)^{n-1}$, where the numerator is
\[
P(S/J(K_n),t)= n+(1-n)(1-t)^{n-1}-\binom{n}{2}t(1-t)^{n-2}.
\]
\noindent The ring $S/J(K_n)$ has Castelnuovo-Mumford regularity two. For $S/J(K_n)^\vee$, the Hilbert series is $P(S/J(K_n)^\vee,t)/(1-t)^{n-1}$, where the numerator is
\[
\!\!\!\!\!\!\!\!\!\!P(S/J(K_n)^\vee,t)= 1-nt^{{n-1 \choose 2}}+{n \choose 2}t^{{n \choose 2}-1}-{n-1 \choose 2}t^{{n \choose 2}},
\]
and $S/J(K_n)^\vee$ has projective dimension three, with $\Z^n$--graded betti numbers 
\begin{equation}
\label{eq:phikbound}
  \begin{array}{ccccc}
    b_{0,{\bf m}} (S/J(K_n)^\vee) &  = & 1 & \mbox{ if } |{\bf m}| = &0, \\
    b_{1,{\bf m}} (S/J(K_n)^\vee) &  = & 1 & \mbox{ if }\ { \bf m}
                                        \leftrightarrow &K_{n-1}
                                                          \subseteq K_n, \\
    b_{2,{\bf m}}(S/J(K_n)^\vee) &  = & 1 & \mbox{ if } |{\bf m}| = &{ n  \choose 2} -1,\\
    b_{3,{\bf m}} (S/J(K_n)^\vee) &  = & { n-1  \choose 2} & \mbox{ if } |{\bf m}|=&{ n  \choose 2}.
   \end{array}                                                 
 \end{equation}
\noindent Theorem \ref{HochsterUSE} describes the entire minimal free resolution for
$S/J(K_n)^\vee$. We discuss Alexander duality in \S 2, prove Theorem A in \S 3, and prove Theorem B in \S 4.
\pagebreak
\section{Alexander duality}\label{AlexDual}
\noindent We will make use of two fundamental results involving Alexander duality: Primary decomposition, and Hochster's theorem. An excellent reference for both is \cite{MillS}.
\subsection{Alexander duality and free resolutions}
\begin{defin}
Fix a field $\k$, and let $\Delta $ be a simplicial complex on vertex
set $V$. If $|V|=n$, let $S = \k[x_1,\ldots, x_n]$. The Stanley-Reisner ring of
$\Delta$ is $S/I_\Delta$,
where
\[
  I_\Delta = \langle x_{i_1}\cdots x_{i_j} \mid [i_1,\ldots, i_j]
  \mbox{ is not a face of }\Delta \rangle
\]
\end{defin}
The ideal $I_\Delta$ encodes all the non-faces of $\Delta$, so in
particular the simplicial complex $\Delta$ and the ideal $I_\Delta$
carry the same information.
\begin{exm}\label{SRexm}
Let $\Delta$ be a simplicial complex on four vertices, and edges $\{[12],[23],[34],[14],[13]\}$. The missing
faces are the edge $[24]$, and all triangles. The missing
triangles $[124]$ and $[234]$ and the missing full $3$-simplex $[1234]$ are consequences of the missing edge, so
\[
  I_\Delta = \langle x_2x_4, x_1x_2x_3, x_1x_3x_4 \rangle.
\]
\end{exm}
\noindent The complement of a face of $\Delta$ is a {\em coface}; let $CF(\Delta)$ denote the set of minimal cofaces of $\Delta$. 
\begin{thm}\label{primaryDecompSR}$[$5.3.3 of \cite{S}$]$ For a simplicial complex $\Delta$, the primary decomposition of $I_\Delta$ is
\[
I_\Delta = \bigcap\limits_{[v_{i_1} \cdots v_{i_k}] \in CF(\Delta)}\langle x_{i_1}, \ldots, x_{i_k} \rangle.
\]
\end{thm}
\begin{exm}\label{SRexm2}In Example~\ref{SRexm} the minimal cofaces of $\Delta$ are $\{[34],[14],[12],[23],[24]\}$, so the primary decomposition is
  \begin{equation}\label{PDSR}
    I_\Delta = \langle x_1,x_2\rangle \cap \langle x_1,x_4\rangle \cap 
    \langle x_2,x_3\rangle \cap \langle x_2,x_4\rangle \cap \langle 
    x_3,x_4\rangle. 
  \end{equation}
\end{exm} 

\noindent  The fact that the minimal generators of a primary component can be chosen as variables is special to squarefree monomial (=Stanley-Reisner) ideals. Notice that by choosing variables as minimal generators of a primary component, the product of the generators of a primary component of a Stanley-Reisner ideal is a monomial. The ideal generated by such monomials (one for each primary component) is called the {\em monomialization} of the primary decomposition of $I_\Delta$.
\begin{defin}
The combinatorial Alexander dual of $\Delta$ is
\[
\Delta^\vee = \{ \sigma \subset V \mid \overline{\sigma} \not\in \Delta\}.
\]
The condition that $\overline{\sigma} \not\in \Delta$ means that $\sigma$ is the complement of a non-face of $\Delta$. 
\end{defin}
\begin{thm}\label{AlexPrimary}$[$\cite{ER} or \cite{Hochster}$]$
Monomializing the primary decomposition of $I_\Delta$ yields $I_{\Delta^\vee}$. 
\end{thm}
\begin{exm}\label{SRexm3}
The nonfaces of $\Delta$ of Example~\ref{SRexm} are
\[
\{[123], [124],[134],[234], [24]\},
\]
so the complements of the nonfaces are 
\[
\{ [4], [3],[2],[1],[13]\}. 
\]
Hence the maximal faces of $\Delta^\vee$ are the vertices $[2],[4]$ and
the edge $[13]$. In particular, all edges are missing from $\Delta^\vee$
except $[13]$. Monomializing the primary decomposition in Equation~\ref{PDSR} yields 
 \[
    I_{\Delta^\vee} = \langle x_1x_2, x_1x_4, x_2x_3, x_2x_4,  x_3x_4\rangle, 
  \]
which is indeed the Stanley-Reisner ideal of $\Delta^\vee$. 
\end{exm}
\subsection{Free resolutions and betti tables}
The Hilbert Syzygy Theorem \cite{e} guarantees that any finitely generated $\Z$-graded $S=\k[x_1,\ldots,x_n]$-module $M$ has a {\em minimal graded finite free resolution}: an exact sequence 
of free modules $F_i \simeq \oplus_{j} S(-j)^{b_{i,j}}$ with $b_{i,j} \in \Z:$ \begin{equation}\label{FFR}
0 \longrightarrow F_i \stackrel{d_i}{\longrightarrow} F_{i-1} \stackrel{d_i}{\longrightarrow} \cdots \longrightarrow F_0 \longrightarrow M \longrightarrow 0, 
\end{equation}
where $i \le n$ and the entries of the $d_i$ matrices are homogeneous of positive degree.
\begin{defin}\label{bettiTable} 
For $M$ as above, the {\em regularity} and {\em projective dimension} are
 \[
  \reg(M) = \sup\{j \mid b_{i,i+j} \ne 0\} \mbox{ and } \pdim(M) = \sup\{i \mid b_{i,\bullet} \ne 0\}.
\]
The {\em graded betti numbers} are
 \[
b_{i,j} = \dim_{\k}Tor_i(M,\k)_j.
\]
\end{defin}
\noindent This data is compactly encoded in the {\em betti table} \cite{e3}: 
an array whose entry in position $(i,j)$ (reading over and down) is $b_{i,i+j}$. This indexing seems odd, but it is set up so that 
$\reg(M)$ is given by the index of the bottom row of the betti table. 
\begin{exm}\label{SR4}
The minimal free resolution for $I_\Delta$ from Example~\ref{SRexm} is given below.
\begin{small}
\[
0 \longleftarrow I_\Delta \xleftarrow{\left[ \!\begin{array}{ccc}
x_2x_4& x_1x_2x_3& x_1x_3x_4\end{array}\! \right]} S(-2) \oplus S(-3)^2
\xleftarrow{\left[ \!
\begin{array}{cc}
x_1x_3 & 0\\
-x_4 & x_4 \\
0 & -x_2
\end{array}\! \right]} S(-4)^2 \longleftarrow 0
\]
\end{small}
\!\!The corresponding betti table is 
\begin{center}
\begin{small}
	\begin{tabular}{c|cc}
		\diagbox{j}{i}&0&1\\
		\hline 
		2&1&--\\
		3&2&2
	\end{tabular}
\end{small}
\end{center}
\!\!\!\! The first column of the table reflects that $I_\Delta$ has one quadratic generator and two
cubic generators. The second column shows there are two syzygies on
the three generators. In the same fashion it is easy to write down the minimal free resolution for $S/I_{\Delta^\vee}$, which has betti table 

\begin{center}
	\begin{small}
		\begin{tabular}{c|cccc}
			\diagbox{j}{i}&0&1&2&3\\
			\hline 
			0&1&--&--&--\\
			1&--&5&6&2
		\end{tabular}
	\end{small}
\end{center}
%
\end{exm}
\noindent Corollary 5.59 of \cite{MillS} shows that $\reg(I_\Delta) = \pdim(S/I_{\Delta^\vee})$; for the example above we have that $\pdim(S/I_{\Delta^\vee})= 3 = \reg(I_\Delta)$. 
\subsection{Hochster's theorem}$[$\cite{Hochster} or \cite{MillS}$]$
As in \S2.2, $\Delta$ is a simplicial complex on
$n$-vertices and $S=\k[x_1,\ldots,x_n]$. However, we now endow $S$
with the $\Z^n$ grading, with $\deg(x_i)={\bf e}_i$.
\begin{thm}\label{HochsterT}
For a simplicial complex $\Delta$ on $n$ vertices and $S$ graded by
$\Z^n$, let ${\bf m} \in \Z^n$ be a multidegree, and $|{\bf m}|=\sum_i
m_i$. Then
  \[
  b_{i, {\bf m}}(I_\Delta) = \dim(\widetilde{H}_{|{\bf 
      m}|-i-2}(\Delta_{\bf m}),\k)
\]
where $\Delta_{\bf m}$ is the subcomplex of $\Delta$ consisting of the
faces of $\Delta$ of weight ${\bf n}$ such that for all $i \in
\{1,\ldots, n\}$, $n_i \le m_i$ (hence ${\bf n} \le {\bf m}$ pointwise).
\end{thm}
\begin{exm}\label{SR5} For Example~\ref{SRexm}, we compute the homology
  groups. The generators of $I_\Delta$ occur in
  multidegrees $\{(0101),(1110),(1011)\}$. We now compute
  \[
\dim(\widetilde{H}_{|{\bf m}|-3}(\Delta_{\bf m}),\k)
  \]
  for ${\bf m} = (1111)$.  For this multidegree, $\Delta_{{\bf m}}$ is
  clearly the entire complex $\Delta$, so
  consists of two (hollow) triangles, sharing the common edge $[13]$,
  hence
  \[
    \dim(\widetilde{H}_{1}(\Delta_{(1111)},\k))=2,
    \]
    yielding two first syzygies on $I_\Delta$. The syzygies themselves are   
    \[
      [x_1x_3,-x_4,0], [x_1x_3,0,-x_2]
    \]
  \end{exm}
  \section{Proof of Theorem A}
Alexander duality led us to Theorem A: it was computations with the quadratic component $J_2$ and the Alexander dual $J_2^\vee$ which indicated that $J_2^\vee$ had a linear resolution. Duality also is central in understanding nets.
\begin{proof}
For the proof of $(1)$, the key is Theorem~\ref{AlexPrimary}: the Alexander dual $I^\vee$ of a squarefree monomial ideal $I$  is obtained by monomializing the primary decomposition of $I$, as in Example~\ref{SRexm3}, combined with the description of the primary decomposition in Theorem~\ref{primaryDecompSR}.
\begin{itemize}
\item We first show
\[
(\Pi,X)  \mbox{ is a }(k,d) \mbox{ net }\implies (J_{\Pi})_d = (J_X^\vee)_d 
\]
\noindent A component of the primary decomposition of the monomial ideal $J_\Pi$ will contain exactly one variable from each block $\pi_i$ of $\Pi$.
Since $(\Pi,X)$ is a net, this means that $J_X \subseteq J_{\Pi}^\vee$, so dualizing yields
\[
J_{\Pi} \subseteq J_X^\vee.
\]
Therefore $(J_\Pi)_d \subseteq (J_X^\vee)_d$; note that it is not true that $J_\Pi = J_X^\vee$. From the definition of $J_\Pi$, $\dim_{\k}(J_{\Pi})_d = k$, so to prove equality it suffices to show that $\dim_{\k}(J_X^\vee)_d = k$. 

Because $J_X$ is a squarefree monomial ideal, it is the Stanley-Reisner ideal of a simplicial complex $\Delta$.
By Theorems~\ref{primaryDecompSR} and \ref{AlexPrimary}, to find the minimal degree generators of $J_X^\vee$, we need to find the biggest faces of $\Delta$.
 As the monomials of $J_X$ correspond to the nonfaces of $\Delta$, the biggest faces of $\Delta$ correspond to monomials which are not divisible by any monomial in $J_X$. 
 As soon as a monomial is divisible by at least one variable from each of the $k$ blocks of $\Pi$, the net condition means it might be in $J_X$. However, the primary decomposition of $J_\Pi$ will have $d^k$ components, whereas $J_X$ has $d^2$.
 
We now argue that the maximal faces of $\Delta$ are exactly the complements of single blocks $\pi_i$; to illustrate, in Example~\ref{braidarr}, the maximal faces of $\Delta$ are 
$\{[1245]  =  \overline{ [36] }, [1346]  =  \overline{ [25] }, [2356]  =  \overline{ [14] } \}$. To see this, notice that a set of lines $\sigma$ corresponds to a non-face of $\Delta$ exactly when 
\[
\sigma \cap \pi_i \ne \emptyset \mbox{ for all blocks } \pi_i.
\]
\noindent The maximal sets which fail to have this property are the complements of a single block of $\Pi$, and the result follows.
\vskip .05in
\item For the other direction
\[(J_{\Pi})_d = (J_X^\vee)_d \implies (\Pi,X) \mbox{  a }(k,d) \mbox{ net}
\]
\noindent For $(\Pi,X)$ a potential $(k,d)$ net, since $(J_{\Pi})_d = (J_X^\vee)_d$ and $J_\Pi$ is generated in degree $d$, we have 
\[
J_{\Pi} \subseteq J_X^\vee,
\]
hence $J_X \subseteq J_\Pi^\vee$. As noted above, $J_\Pi^\vee$ is generated by monomials obtained by taking exactly one element from every block, so this means every generator of $J_X$ satisfies this property, hence so also does $X$.
\end{itemize}
\noindent For the proof of $(2)$, the assumption on the net $(\Pi, X)$ means that if (after a change of variables) $\pi = |1,\ldots, d|$ is a block of the
 partition $\Pi$, then for $1\le i < j \le d$, $H_{i}$ and  $H_{j}$ meet in a normal crossing point $p$. Thus, for each block $\pi$ as above, we have a subideal of $J$
 \[
   J_\pi = \langle x_{i}x_{j} \mid 1 \le i < j \le d \rangle.
 \]
 For generic $\{a_1,\ldots,a_d\}$, $J_\pi$ can be written as the ideal generated by the $2 \times 2$ minors of the $2 \times
 d$ matrix
 \begin {center}
 $ M=\left[ \!
\begin{array}{cccc}
  x_1& x_2  &\cdots &x_d   \\
a_1x_1& a_2x_2  &\cdots &a_dx_d  
                             \end{array}\! \right], $
\end{center}
which by Theorem A2.10 of \cite{e} has an Eagon-Northcott resolution
\[
  \cdots  \longrightarrow Sym^1(S^2)^* \otimes \Lambda^3(S^d) \longrightarrow \Lambda^2(S^d) \stackrel{\Lambda^2(M)}{\longrightarrow}
\Lambda^2(S^2) \longrightarrow S/I_2(M) \longrightarrow 0. 
\]
As the variables in the blocks of $\Pi$ are distinct, we see that
\[
  J_2 \simeq \bigoplus\limits_{\pi_i \in \Pi} J_{\pi_i},
\]
The hypothesis on the quadratic component $J_2$ means that there is a partition $\Pi$ of the $d\cdot k$ hyperplanes into $k$-blocks of size $d$, and that all hyperplanes within a block have normal crossing intersection. Therefore, any point of intersection with multiplicity greater than two cannot be contained in a block of $\Pi$, so must lie in $X$. The condition on the primary decomposition from $(1)$ ensures that every multiple point in $X$ meets exactly one line from each block of $\Pi$.
\end{proof}
\begin{remark}
We thank an anonymous referee for suggesting a simplification in the proof of $(1)$. 
\end{remark}

\begin{exm}\label{CevaArr}
An infinite family of $(3,d)$ nets is the Ceva family, given by the arrangement defined by the vanishing of the polynomial $(x^d-y^d)(x^d-z^d)(y^d-z^d)$. Note that the vanishing set of $x^d-y^d$ defines $d$ lines passing thru the point $(0:0:1)$, and similarly for the vanishing sets of $x^d-z^d$ and $y^d-z^d$. This yields a $(3,d)$ net
\[
| 1,\ldots, d | d+1,\ldots 2d | 2d+1,\ldots,3d|
\]
Therefore $J$ is generated by the three polynomials of degree $d$ $($which define $J_\Pi)$
\[
\{x_1 \cdots x_d, x_{d+1} \cdots x_{2d}, x_{2d+1} \cdots x_{3d}\},
\]
and $d^2$ cubics corresponding to the triple intersections. As noted in \cite{harbourne}, for complex line arrangements in $\P^2$, this is the only infinite family known to have no normal crossing points. 
\noindent For $d=3$ the matroid is depicted below, where points denote the lines of the configuration in $\P^2$. The set $X$ consists of $\{ \{147\},\{258\},\{369\},\{168\},\{249\},\{267\},\{348\},\{159\},\{357\}\}$. 
\begin{figure}[h]
\vskip -1.4in
\includegraphics[height=4in, width=2.5in]{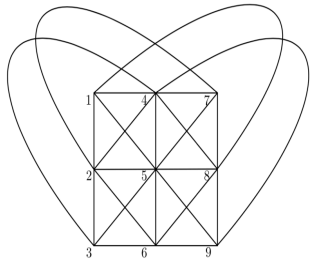}
\vskip -1.5in
\caption{The Ceva net for $d=3$}
\end{figure}
\noindent In \cite{Bar}, Bartz gives a classification of complete 3-nets. As noted in \cite{harbourne}, there are also two isolated examples of configurations in $\P^2_{\C}$ with no normal crossings: Klein's configuration has 21 lines, and Wiman's configuration has 45 lines. It would be interesting to check if these configurations support nets.
\end{exm}
\subsection{An aside on the singular fibers of a net or pencil}
In classical algebraic geometry, the term {\em net} refers to a three dimensional subspace of the space of sections of some line bundle on an algebraic variety $Z$, while a {\em pencil} is a two-dimensional subspace. 
Therefore a net gives a rational map $\phi$ from $Z$ to $\P^2$, and a pencil gives a rational map $\phi$  from $Z$ to $\P^1$. 
\begin{defin}$[$3.3, \cite{FY}$]$ If $F_1, F_2 \in \C[x,y,z]_d$ have no common factor, then the pencil $a_1F_1+a_2F_2$ with $[a_1:a_2] \in \P^1$ is {\em Ceva type} if there are three or more fibers that factor as products of linear forms, and after blowing up the base locus, the proper transforms of every fiber of $\phi$ are connected. 
\end{defin} 
\begin{prop}\label{singFiber}
If $\A=(\Pi, X)$ is a net satisfying part (2) of Theorem A (all intersections of lines within a block are normal crossing), then the net has singular fibers beyond the singular fibers coming from the blocks of $\Pi$,   unless $\A$ is a $(3,2)$ net or a $(4,3)$ net.
\end{prop}
\begin{proof}
In Theorem 4.2 of \cite{FY}, Falk-Yuzvinsky prove a result on the numerics of the Euler characteristic of a multinet $(\Pi,X)$, which for a $(k,d)$ net takes the form
\begin{equation}\label{FYeqn}
3 + |X| \ge (2-k)(3d-d^2) + 2kd- \sum\limits_{p \in \overline{X}} \mu(p), 
\end{equation}
\noindent where $\overline{X}$ is the set of points of intersection of $\A$ of the lines {\em within} the blocks $\pi_i$. They prove that equality holds in Equation~\ref{FYeqn} iff the only singular fibers of the net are the blocks. Since $|X|=d^2$ for a net, if all intersections within the blocks have $\mu(p)=1$, then
\[\sum\limits_{p \in \overline{X}} \mu(p) = k{d \choose 2}.
\]
 Equality in Equation~\ref{FYeqn} is only possible if $k = 6 -\frac{6}{d}$, which implies $(k,d) \in \{(3,2),(4,3),(5,6)\}$, and by \cite{Y09} only $(3,2)$ and $(4,3)$ can arise for a net. 
\end{proof}

\noindent The next example illustrates both Theorem A and Proposition~\ref{singFiber}. 
  \begin{exm}
    The Pappus and non-Pappus arrangements appear as examples 9 and 10
    in Suciu's survey paper \cite{Su}. Both are arrangements of 9
    lines in $\P^2$; each has 9 double points and 9 triple points. There is a ninth line (not pictured) at infinity.
    
\begin{figure}[h] 
\vskip -.5in
\includegraphics[trim = 45mm 90mm 45mm 90mm, clip,width=55mm,height=45mm]{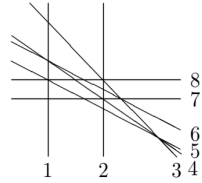}
\vskip -.5in
\caption{non-Pappus arrangement}
\end{figure}
\noindent The non-Pappus arrangement does not support a net. The betti table of $S/J_2$ is

\begin{center}
	\begin{small}
		\begin{tabular}{c|ccccccc}
			\diagbox{j}{i}&0&1&2&3&4&5&6\\
			\hline 
			0&1&--&--&--&--&--&--\\
			1&--&9&9&--&--&--&--\\
			2&--&--&18&18&--&--&--\\
			3&--&--&--&3&9&9&2
		\end{tabular}
	\end{small}
\end{center}
%
\noindent The Pappus configuration below supports a $(3,3)$ net, with blocks $|169|258|347|$
\begin{figure}[h]
\vskip -.3in
\includegraphics[trim = 45mm 90mm 45mm 90mm, clip,width=55mm,height=45mm]{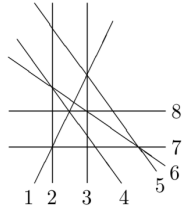}
\vskip -.5in
\caption{Pappus arrangement}
\end{figure}

\noindent For the Pappus arrangement, the cubic generators of the monomial OS ideal are
\[
\begin{array}{c}
\{x_1x_2x_7,x_1x_3x_5, x_1x_4x_8,x_2x_3x_9,x_2x_4x_6,x_3x_6x_8,x_4x_5x_9,x_5x_6x_7, x_7x_8x_9\}.
\end{array}
\]
\noindent The primary decomposition of $J_\Pi = \langle x_1x_6x_9,x_2x_5x_8,x_3x_4x_7\rangle$ has components:
\[
\begin{array}{c}
\left(x_{1}, x_{2}, x_{3}\right),\left(x_{1}, x_{2}, x_{4}\right),\left(x_{1}, x_{2}, x_{7}\right),\left(x_{1}, x_{3}, x_{5}\right),\left(x_{1}, x_{3}, x_{8}\right),\left(x_{1}, x_{4}, x_{5}\right),\\
\left(x_{1}, x_{4}, x_{8}\right),\left(x_{1}, x_{5}, x_{7}\right), \left(x_{1}, x_{7}, x_{8}\right),\left(x_{2}, x_{3}, x_{6}\right),\left(x_{2}, x_{3}, x_{9}\right),\left(x_{2}, x_{4}, x_{6}\right),\\
\left(x_{2}, x_{4}, x_{9}\right),\left(x_{2}, x_{6},  x_{7}\right),\left(x_{2}, x_{7}, x_{9}\right),\left(x_{3}, x_{5}, x_{6}\right),\left(x_{3 }, x_{5}, x_{9}\right),\left(x_{3}, x_{6}, x_{8}\right),\\
\left(x_{3}, x_{8}, x_{9}\right)  ,\left(x_{4}, x_{5}, x_{6}\right),\left(x_{4}, x_{5}, x_{9}\right),\left(x_{4}, x_{6}, x_{8}\right)   ,\left(x_{4}, x_{8}, x_{9}\right),\left(x_{5}, x_{6}, x_{7}\right),\\
\left(x_{5}, x_{7}, x_{9}\right),\left(x_{6}, x_{7}, x_{8}\right),\left(x_{7}, x_{8}, x_{9}\right)
\end{array}
\]
\pagebreak

\noindent So the cubic minimal generators of $J$ are elements of $J_\Pi^\vee$. The betti table of $S/J_2$ is given by

\begin{center}
	\begin{small}
		\begin{tabular}{c|ccccccc}
			\diagbox{j}{i}&0&1&2&3&4&5&6\\
			\hline 
			0&1&--&--&--&--&--&--\\
			1&--&9&6&--&--&--&--\\
			2&--&--&27&27&12&--&--\\
			3&--&--&--&27&54&36&8
		\end{tabular}
	\end{small}
\end{center}

\noindent and the betti table of the Alexander dual $S/J_2^\vee$ is 

\begin{center}
	\begin{small}
		\begin{tabular}{c|ccccc}
			\diagbox{j}{i}&0&1&2&3&4\\
			\hline 
			0&1&--&--&--&--\\
			1&--&--&--&--&--\\
			2&--&--&--&--&--\\
			3&--&--&--&--&--\\
			4&--&--&--&--&--\\
			5&--&27&54&36&8
		\end{tabular}
	\end{small}
\end{center}

 \noindent In contrast, for the Alexander dual $S/J_2^\vee$ of the non-Pappus arrangement the betti table is

\begin{center}
	\begin{small}
		\begin{tabular}{c|ccccc}
			\diagbox{j}{i}&0&1&2&3&4\\
			\hline 
			0&1&--&--&--&--\\
			1&--&--&--&--&--\\
			2&--&--&--&--&--\\
			3&--&--&--&--&--\\
			4&--&9&9&--&--\\
			5&--&3&9&9&2
		\end{tabular}
	\end{small}
\end{center}
\noindent As the ideal $J_2^\vee$ for the Pappus arrangement has a linear resolution, by the Eagon-Reiner theorem \cite{ER}, $S/J_2$ is Cohen-Macaulay, while the dual ideal $J_2^\vee$ of the non-Pappus arrangement does not have a linear resolution, so for the non-Pappus arrangement, $S/J_2$ is not Cohen-Macaulay.
\end{exm}

\begin{exm}
By Proposition~\ref{singFiber}, the Pappus arrangement has singular fibers in addition to the three normal crossing sets of lines of the blocks of $\Pi$. The unique $(3,2)$ net is that of Example~\ref{braidarr}, and the only known $(4,3)$ net is the Hessian; both satisfy the hypotheses of part (2) of Theorem A.  
\end{exm}
\section{The monomial Orlik-Solomon algebra for $A_{n-1}$}
 Let $J(K_n)$ denote the ideal $J$ for the braid arrangement $A_{n-1}$ (equivalently, the graphic
 arrangement $K_n$). Our focus in this section
 is on the algebraic behavior of the ideals $J(K_n)$ and $J(K_n)^\vee$. 
 To prove Theorem B, we first prove the projective dimension
of $S/J(K_n)^\vee$ is three, which follows from an analysis of the
corresponding ideal. With the bound on projective dimension of
$S/J(K_n)^\vee$ in hand, an analysis using Theorem~\ref{HochsterT}
yields the betti numbers for $S/J(K_n)^\vee$. 

Applying Corollary 5.59 of \cite{MillS} shows that $J(K_n)$
has regularity two.  As  there are only two rows in the betti table of
$J(K_n)$, to determine the betti numbers, it suffices to determine
the Hilbert series of $J(K_n)$ and one row of the betti table. 

The top row
of the betti table of $J(K_n)$ corresponds to a squarefree ideal generated by
quadrics, so is an edge ideal of a graph $\Gamma$. The linear strand of the resolution is
interesting in its own right, as it is depends on the cut polynomial of $\Gamma$.

\subsection{Hilbert series of $J(K_n)$} We begin by describing the
generators of the monomial ideals $J(K_n)$ and $J(K_n)^\vee$. For the complete
graph $K_{n}$, the rank one elements of $L(A_{n-1})$ are the ${n
  \choose 2}$ hyperplanes $V(x_i-x_j)$, so the rank two elements
correspond to
    \begin{enumerate}
\item Triangles in $K_n$.
\item Pairs of disjoint edges in $K_n$.
\end{enumerate}
We study the Stanley-Reisner ring in ${n \choose 2}$
variables $S=\k[x_{ij}\mid 1\le i < j \le n]$, modulo the ideal 
\[
  J(K_n) = \langle x_{ij}x_{kl}, x_{ij}x_{ik}x_{jk} \mid i,j,k,l \mbox{
    distinct }\in \{1,\ldots,n\} \rangle.
\]
\pagebreak
\begin{example}
\noindent For $K_4$, the minimal non faces of $\Delta_{J(K_4)}$ are 
\[
\{[12][34],[13][24],[14][23], [12][14][24],[23][24][34],[13][14][34],[12][13][23]\}
\]
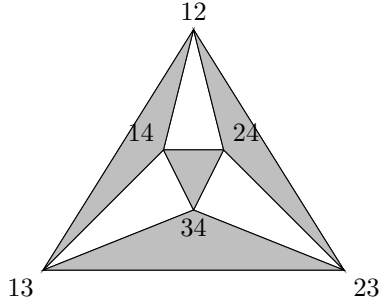
\begin{figure}[h]
\vskip -.2in

\begin{tikzpicture}[scale=0.8]
\draw (0,0) -- (5,0);
\draw (0,0) -- (2.5,4);
\draw (5,0) -- (2.5,4);
\draw (2,2) -- (3,2);
\draw (2,2) -- (2.5,1);
\draw (2.5,1) -- (3,2);
\draw (0,0) -- (2.5,1);
\draw (0,0) -- (2,2);
\draw (2,2) -- (2.5,4);
\draw (3,2) -- (2.5,4);
\draw (5,0) -- (3,2);
\draw (5,0) -- (2.5,1);
\draw[fill=gray!50] (0,0) -- (5,0) -- (2.5,1) -- cycle;
\draw[fill=gray!50] (3,2) -- (2,2) -- (2.5,1) -- cycle;
\draw[fill=gray!50] (0,0) -- (2,2) -- (2.5,4) -- cycle;
\draw[fill=gray!50] (3,2) -- (5,0) -- (2.5,4) -- cycle;
\draw (2.5,4) node [anchor=south]{12};
\draw (0,0) node [anchor=north east]{13};
\draw (5,0) node [anchor=north west]{23};
\draw (2.5,1) node [anchor=north]{34};
\draw (2,2) node [anchor=south east]{14};
\draw (3,2) node [anchor=south west]{24};
\end{tikzpicture}
\vskip -.2in
\caption{The simplicial complex $\Delta_{J(K_4)}$}
\label{deltaK4}
\end{figure}

\noindent The complex in Figure \ref{deltaK4} corresponds to the braid arrangement in Figure \ref{fig:braid}, where the vertices 12, 23, 13, 34, 14 and 24 in Figure 5 correspond, respectively, to the lines $L_1,\dots,L_6$.

\end{example}

\begin{lemma}\label{DeltaDes}
The simplicial complex $\Delta_{J(K_n)}$ corresponding to the ideal $J(K_n)$
consists of $n$ simplices of dimension $n-2$, each meeting the others
in $n -1$ points, and the face vector of $\Delta_{J(K_n)}$ is 
given by (with notation as in \cite{Z}, \S 8.3,) $f(\Delta_{J(K_n)},x) =$
\[
 x^{n-1} + n\cdot\Big((n-1)x^{n-2}+{n-1 \choose 2}x^{n-3}+{n-1 \choose 3}x^{n-4}+\ldots+(n-1)x+1\Big) - {n \choose 2}x^{n-2}
\]
\end{lemma}
\begin{proof}
Vertices in $\Delta_{J(K_n)}$ correspond to edges in $K_n$. The maximal faces of  $\Delta_{J(K_n)}$ correspond to simple graphs with vertex set $[n]$ that have no pairs of the types (1) or (2) above. The maximal such graphs are clearly the $n$ star graphs with $n-1$ edges,
\[
  \Delta_i = \{[i,j_1],\ldots, [i,j_{n-1}]\} \mbox{ for a fixed }i.
\]  
Every pair of such graphs share a common edge, and every edge lies in exactly two such graphs. This explains the term ${n\choose 2} x^{n-2}$. 
\end{proof}
\begin{remark}
	As pointed out by the referee, this complex is the nerve of the cover of $K_n$ by closed 	edges, homotopy equivalent to $K_n$.
\end{remark}
 
\begin{exm}
  For $K_7$, the f-vector is 
  \[
\begin{array}{ccc}
    f(\Delta_{J(K_n)}) &= &(1,0,0,0,0,0,0)+7(0,6,15,20,15,6,1) -
    (0,21,0,0,0,0,0) \\
    f(\Delta_{J(K_n)}, x)&=& \! \! \! \! \! \! \! \! \! \! \! \! \! \! \! \! \! \! \! \! \! \! \! \! \! \! \! \! \! x^6+21x^5+105x^4+140x^3+105x^2+42x+7.
\end{array}
  \]
\end{exm}
\noindent By \cite{Z}, \S 8.3, 
\[
  h(\Delta, x)=f(\Delta, x-1),
\]
so for $K_n$ we have $h(\Delta_{J(K_n)},x) =$ 
\[
(x-1)^{n-1} +n\Big((n-1)(x-1)^{n-2}+{n-1 \choose 2}(x-1)^{n-3}+\ldots +1\Big) - {n \choose 2}(x-1)^{n-2}
\]
which simplifies to 
$$(x-1)^{n-1}+n(x^{n-1}-(x-1)^{n-1})-{n\choose 2}(x-1)^{n-2}$$

\noindent In \cite{K}, \S 8.3, $h_i(\Delta)$ is the coefficient of $x^{\dim \Delta +1 -i}$ in $h(\Delta,x)$, so since 
$$ P(S/I_\Delta,t)= (h_0+h_1t+h_2t^2+\cdots)/(1-t)^{\dim \Delta +1},$$ we need to reverse the order of the $h-$vector, and the result for the Hilbert series in Theorem B follows. 


\subsection{Resolution of the Alexander dual}
We now turn to the Alexander dual ideal $J(K_n)^\vee$. As noted in \S 2,  $J(K_n)^\vee$ is the
monomialization of the primary decomposition of $J(K_n)$. Letting $CF(\Delta)$ denote the set of
minimal cofaces of $\Delta$, Theorem \ref{primaryDecompSR} yields
\[
  J(K_n) = \bigcap\limits_{[x_{i_1},\ldots, x_{i_k}] \in CF(\Delta)}\langle x_{i1},\ldots, x_{ik}\rangle.
\]
By Lemma~\ref{DeltaDes}, $\Delta$ consists of $n$ copies of the $n-2$
simplex $\Delta_{n-2}$, glued at a total of ${n \choose 2}$ vertices,
so a maximal face of $\Delta$ is one of the $\Delta_{n-2}$'s, whose
complement coface consists of the remaining
\[
  {n \choose 2} -(n-1) = {n-1 \choose 2}
\]
vertices. Therefore $J(K_n)^\vee$ is generated by the corresponding $n$
monomials of degree ${n-1 \choose 2}$.
\begin{thm}\label{HochsterUSE}
The projective dimension of $S/J(K_n)^\vee$ is three, and the $\Z^{ n
  \choose 2}$graded betti numbers are
\[
  \begin{array}{ccccc}
    b_{0,{\bf m}} (S/J(K_n)^\vee) &  = & 1 & \mbox{ if } |{\bf m}| = &0 \\
    b_{1,{\bf m}} (S/J(K_n)^\vee) &  = & 1 & \mbox{ if }\ {\bf m}
                                        \leftrightarrow &K_{n-1}
                                                          \subseteq K_n \\
    b_{2,{\bf m}}(S/J(K_n)^\vee) &  = & 1 & \mbox{ if } |{\bf m}| = &{ n  \choose 2} -1\\
    b_{3,{\bf m}} (S/J(K_n)^\vee) &  = & { n-1  \choose 2} & \mbox{ if } |{\bf m}|=&{ n  \choose 2}
   \end{array}                                                 
  \]
  \end{thm}
  \begin{proof}
    The proof follows from Hochster's formula. For each $K_{n-1}$ in
    $K_n$, there is a generator with weight vector having
    entry $1$ in the positions corresponding to edges of the
    $K_{n-1}$, hence $n$ generators of weight ${n-1 \choose 2}$.

   We now make use of the LCM lattice resolution, as described in
   \cite{GPW}. The first syzygies of a monomial ideal correspond to the
    LCM of two generators. The two corresponding $K_{n-1}$
    subgraphs $\Delta_i$ and $\Delta_j$  intersect everywhere except
    along the missing edge $[ij]$ so the LCM corresponds to a weight vector which is
    one in all but the single entry $[ij]$. This yields ${n \choose 2}$ first
    syzygies, with weight as above.

    For the second syzygies, by Hochster's formula they must
    have weight $(1,1,\ldots, 1)$, and correspond to the LCM of triples of
    monomials, of which there are ${n \choose 3}$. However, these
    choices are not independent, since every ${n \choose i}$ set with
    $i \ge 4$ has LCM of weight $(1,1,\ldots, 1)$. So accounting for
    dependencies, dependencies on dependencies, and so on, we find
    that the number of minimal second syzygies is
    \[
{n \choose 3}-{n \choose 4}+{n \choose 5}-{n \choose 6}+ \cdots = 1 -n
+{n \choose 2} = {n-1 \choose 2},
\]
which concludes the proof.
\end{proof}
\noindent Combining Lemma~\ref{DeltaDes}, Theorem~\ref{HochsterUSE}, and Corollary 5.59 of \cite{MillS} proves Theorem B.
  \begin{example}\label{K7betti}
    The $\Z$-graded betti table for $S/J_7^\vee$ is
   
   \begin{center}
   	\begin{small}
   		\begin{tabular}{c|cccc}
   			\diagbox{j}{i}&0&1&2&3\\
   			\hline 
   			0&1&--&--&--\\
   			1&--&--&--&--\\
   			\vdots&\vdots&\vdots&\vdots&\vdots\\
   			13&--&--&--&--\\
   			14&--&7&--& --    \\
   			15&-- &--&--& --    \\
   			16&-- &--&--& --    \\
   			17&--&--&--& --    \\
   			18&-- &--&21& 15 
   		\end{tabular}
   	\end{small}
   \end{center}
The Hilbert series of $P(R/J_7^\vee,t)$ is therefore
\[
 \frac{1-7t^{15}+21t^{20}-15t^{21}}{(1-t)^6} =
  \frac{\big(\sum\limits_{i=0}^{14}(i+1)t^i\big) +9t^{15}+3t^{16}-3t^{17}-9t^{18}-15t^{19}}{(1-t)^4}.
\]
    \end{example}
\subsection{The linear strand of $J(K_n)$ and the cut polynomial}
A squarefree quadratic monomial ideal encodes the edges of a graph $\Gamma$,
with a generator $x_ix_j$ corresponding to the edge $[ij]$; such
ideals are often called {\em edge ideals}, and there is a wide
literature on the topic; see \cite{RV}. It follows from Hochster's theorem that the betti numbers $b_{i,i+1}(S/I_\Gamma)$ are determined by the 
{\em cut polynomial} of the graph $\Gamma$; this is used by Papadima-Suciu in \cite{PS} to establish a formula for the Chen ranks of right angled Artin groups; their result is over $E$, but it translates to $S$.
\begin{defin}
For a simple (no loops or multiple edges) graph $\Gamma$, the cut polynomial is defined via
\vskip -.1in
\[
Q_\Gamma(t) = \sum\limits_{j \ge 2} c_j(\Gamma)t^j, \mbox{ where } c_j(\Gamma) = \sum\limits_{\stackrel{W \subseteq V}{|W| = j}}(|\mbox{connected components of }\Gamma_W|-1).
\]
\end{defin}
\noindent In \cite{Hochster}, Hochster proves the betti numbers $b_{j,j+1}$ of $S/I_\Gamma$ satisfy
\[
b_{j,j+1}(S/I_\Gamma) = c_{j+1}(\Gamma).
\]
Since the regularity of $S/J(K_n)$ is two and we know the Hilbert series, to determine the betti table, it suffices to determine the top row, hence to finding the coefficients of the cut polynomial of the graph $\Gamma_n$ corresponding to the quadratic generators of $J(K_n)$. Let $I_{\Gamma_n}$ denote the corresponding edge ideal. For small $n$ the $c_k$ are: 
\begin{table}[ht]
\begin{tabular}{|c|c|c|c|c|c|c|c|c|c|c|c|}
\hline $n$& $c_2$ & $c_3$ & $c_4$ & $c_5$ & $c_6$ & $c_7$ & $c_8$ & $c_9$ & $c_{10}$ & $c_{11}$& $c_{12}$\\
\hline $4$& $3$ & $0$ & $0$ & $0$ & $0$ & $0$ & $0$ & $0$ & $0$ & $0$& $0$\\
\hline $5$& $15$ & $30$ & $10$ & $0$ & $0$ & $0$ & $0$ & $0$ & $0$ & $0$& $0$\\
\hline $6$& $45$ & $210$ & $390$ & $285$ & $100$ & $15$ & $0$ & $0$ & $0$ & $0$& $0$\\
\hline $7$& $105$ & $840$ & $3150$ & $6510$ & $7497$ & $5565$ & $2835$ & $980$ & $210$ & $21$& $0$\\
\hline
\end{tabular}
\end{table}
\vskip -.01in
\noindent As any vertex $v \in \Delta_{J(K_n)}$ lies on a pair of $n-2$ simplices, to disconnect $v$ requires removing the $2(n-2)$ vertices adjacent to $v$, leaving a total of ${n \choose 2}-(2n-3)$ vertices, hence $c_j(J(K_n))$ vanishes when $j \ge {n-2 \choose 2}+2$. 
\begin{problem} Determine the cut polynomial for $\Gamma_n$. It is not hard to show that the first two $c_i$ are
\[
\begin{array}{ccc}
c_2=b_{12}(S/I_{\Gamma_n}) & = &\!\!\!\!\!\!\!\!\!\! \!\!\!\!\!\!\!\!\!\!\!\!\!\!\!\!\!\!\!\! \!\!\!\!\!\!\!\!\!\!\!\!\!\!\!\!\!\!\!\!\!3 {n \choose 4}\\
c_3=b_{23}(S/I_{\Gamma_n}) & = & 3 {n \choose 2,3,n-5} + \frac{1}{3} {n \choose 2,2,2,n-6}
\end{array}
\]
\end{problem}
\section{Appendix: The Orlik-Solomon algebra and Resonance Varieties}
\noindent For a hyperplane arrangement
\vskip .05in
\begin{center}
  $\A = \bigcup\limits_{i=1}^m H_i \subseteq V\simeq \kk^{\ell+1},  $
\end{center}
\vskip .05in
we write $M$ for the complement $V \setminus \A$; unless otherwise
noted $V$ is a $\kk= \C$ vector space. We focus on the case where $\A \subset V$  is {\em central} and
{\em essential}, where central means the linear forms defining the
$H_i$ are homogeneous, and essential means the common intersection of
the $H_i$ is $0 \in V$.  Note that $\A$ defines both an {\em affine} arrangement in $V$, as well as a projective arrangement in $\P(V)$.

Orlik and Solomon prove in \cite{OS} that the cohomology ring
$H^*(M,\Z)$ has a purely combinatorial description: it is determined by the intersection lattice 
$L(\A)$. This lattice (in the graded poset sense) consists 
of the intersections of elements of $\A$, ordered by reverse
inclusion. The ambient vector space $V = \hat{0}$, rank one 
elements of $L_{\mathcal A}$ are hyperplanes, and the origin is
$\hat{1}$.

For the complement $M$ of a hyperplane arrangement $\A$
with fundamental group $\pi_1(M)=G$, the first resonance variety $R^1(G)$ is the jump locus for the cohomology of $G$. 
We work over a field $\k$ of characteristic zero; by Orlik-Solomon's result $H^*(M,\k) \simeq E/I_\A$ (see Definition~\ref{formalDefOS}). By convention we write $A$ for $H^*(M,\k)$ and $R^1(A)$ for $R^1(G)$. $A$ is $\Z$-graded; as it is a cohomology ring we denote the $i^{th}$ graded component by $A^i$ . By Falk \cite{Fa}, the  points $(a_1,\ldots,a_m)$ of $R^1(A)\subseteq \k^m$ correspond to one-forms  $a=\sum_1^m a_ie_i$ where the map $\lambda \mapsto \lambda \wedge a$ from $A^1$ to $A^2$ has rank $\le m-2$. 

\subsection{Combinatorics of Arrangements}\label{combinatorics} Two important combinatorial players are the M\"{o}bius function and  Poincar\'e polynomial:
\begin{defin}
The M\"{o}bius function $\mu$ : $L(\A) \longrightarrow \Z$ is given by
$$\begin{array}{ccc}
\mu(\hat{0}) & = & 1\\
\mu(b) & = & -\!\!\sum\limits_{a < b}\mu(a) \mbox{, if } \hat{0}< b.
\end{array}$$

\noindent The Poincar\'e polynomial $\pi(\A,t) = \!\!\sum\limits_{x \in L({\mathcal A})}\mu(x) \cdot (-t)^{\text{rank}(x)}$, and is equal to $ \sum_{i=1}^{\ell+1}(\dim_{\k} A^i) t^i$.
\end{defin}
\noindent In \cite{Fa} Falk introduced the concept of a {\em neighborly partition}:
\begin{defin}\label{NP}\cite{Fa}
A partition $\Pi$ of a subset $U$ of $\A$ is neighborly if
for every codimension two intersection $H_{i_1} \cap \cdots \cap
H_{i_k}$ with $\{i_1,\ldots i_k\} \subseteq U$ if all but one of the $i_j$ are contained in a block of
$\Pi$, then $\{i_1,\ldots, i_k\}$ is contained in the block. 
\end{defin}
\begin{example}
For  the arrangement $A_3$ in Example \ref{braidarr}, the partition $\{14|25|36\}$ is a neighborly partition of the set $\{1,2,3,4,5,6\}$ of all lines of the arrangement.
\end{example}
\subsection{Algebra of Arrangements}\label{algebra}
The central algebraic object in the study of hyperplane arrangements
is the cohomology ring of the arrangement complement, which was
described by Orlik-Solomon in their landmark paper \cite{OS}.
\begin{defin}\label{formalDefOS}
  The Orlik-Solomon algebra $A$ with coefficients in $\k$ of an arrangement
  \[
   \A=\bigcup\limits_{i=1}^m H_i \subseteq \P^\ell
  \]
 is the quotient of the exterior algebra $E=\bigwedge (\k^m)$ on generators $e_1, \dots , e_m$
in degree $1$ by the ideal $I_\A$ generated by all elements of
the form
\[
  \partial(e_{i_1} \cdots e_{i_r})=\sum_{q}(-1)^{q-1}e_{i_1} \cdots
\widehat{e_{i_q}}\cdots e_{i_r}, \mbox{ such that }\codim H_{i_1}\cap
\cdots \cap H_{i_r} < r.
\]
The ideal $I_\A$ is generated in degree $\ge 2$, so $A^0=\k$ and
$A^1=\k^m$.
\end{defin}
\begin{exm} For Example~\ref{braidarr}, $ I_\A = \langle \partial(e_1e_2e_3), \partial(e_1e_5e_6),    \partial(e_2e_4e_6),\partial(e_3e_4e_5)\rangle$. 
 \end{exm}
\noindent It is clear from the definition that the Orlik-Solomon algebra $A$ depends purely on the combinatorics of $L_{\A}$, as do the neighborly partitions appearing in Definition~\ref{NP}. Nets are connected to the resonance variety $R^1(A)$ via the results of \cite{FY}.
As noted earlier, for each $a\in A^1$, we have $a \wedge a=0$, so
exterior right-multiplication by $a$ defines a cochain complex of
$\k$-vector spaces
\begin{equation} \label{eq:aomoto}
(A,a)\colon \quad
\xymatrix{
0 \ar[r] &A^0 \ar[r]^{\wedge a} & A^1
\ar[r]^{\wedge a}  & A^2 \ar[r]^{\wedge a}& \cdots \ar[r]^{\wedge a}
& A^{k}\ar[r]^{\wedge a} & \cdots}.
\end{equation}
The complex $(A,a)$ was introduced by Aomoto~\cite{Ao}, and used by Esnault-Schechtman-Viehweg \cite{ESV} and Schechtman-Terao-Varchenko \cite{STV}
to study local system cohomology.  For $a=\sum_{i=1}^n a_ie_i \in A^1$, we can define the loci where there complex is not exact:
\[
R^j(A) = \{a \in A^1 \mid \dim H^j(A,a) \ne 0\},
\]
\vskip .01in
\noindent which are homogeneous algebraic subvarieties of $\P^{m-1}$, introduced by Falk in \cite{Fa}. Yuzvinsky shows in \cite{Yuz} that the complex is exact as long as $\sum_{i=1}^n a_i \ne 0$.
Falk shows that each component of $R^1(A)$ is associated to a neighborly partition of a subset $U$ of the hyperplanes of $\A$. 
 In \cite{FY} Falk-Yuzvinsky give an interpretation of this in terms of the geometry of {\em multinets} in $\P^2$, which are collections of lines with multiplicity.
\begin{exm}\label{resExample} For Example~\ref{braidarr}, since $e_ie_j-e_ie_k+e_je_k = (e_i-e_j)(e_j-e_k)$, the generators of $I_{\mathcal A}$ can be written 
  \[
  \begin{array}{c}
                                           (e_1-e_2)\wedge(e_2-e_3),\\
                                           (e_1-e_5)\wedge(e_5-e_6), \\
                 (e_2-e_4)\wedge(e_4-e_6), \\
              (e_1-e_3)\wedge(e_3-e_5).
                   \end{array}                
                 \]
                 Hence for exterior multiplication with the one form $a$
                 \[
                   A^1 \stackrel{\wedge  a}{\longrightarrow} A^2,\]
                   the $\P^1$ spanned by $\{e_1-e_2,
                   e_2-e_3\} \subseteq \P^5\simeq \P(A^1)$ is a component of
                   $R^1(A)$, as are the three $\P^1$'s associated to
                   the other three quadrics in $I_{\A}$. A computation shows that 
                   \[
                   (e_1-e_2+e_4-e_5)\wedge(e_2-e_3+e_5-e_6) = \partial(e_1e_2e_3)+\partial(e_1e_5e_6)-\partial(e_2e_4e_6)+\partial(e_3e_4e_5).
                   \]
                   This means there is a fifth $\P^1$ in $R^1(\A)$. As shown by Falk in \cite{Fa},  $R^1(\A)$ is the union of these five lines. The fifth component comes from a neighborly partition $\Pi=|14|25|36|$. As noted in Example~\ref{braidarr} the partition $\Pi$ is a $(3,2)$ net.
\end{exm}
\noindent  For a maximal subset $U =\{H_{i_1},\ldots, H_{i_k}\}$ having codimension two intersection, the
partition into singleton blocks $|i_1|i_2|\cdots|i_k|$ is
neighborly, and yields components of $R^1(\A)$ which are called {\em local}. These components correspond to generators of $I_\A$, as in the first four $\P^1$'s in Example~\ref{resExample}. In \cite{CScv} and \cite{LY} it is shown that $R^1(\A)$ is a union of projectively disjoint projective subspaces. 

 \subsection{Acknowledgements} Our collaboration began at the CIRM conference on ``Lefschetz Properties in Algebra, Geometry
  and Combinatorics''; we thank CIRM and the organizers. Computations in {\tt Macaulay2} \cite{GS} were essential to our
  work. We thank two referees for detailed and useful suggestions, especially a nice simplification in the proof of Theorem A.
  
\bibliographystyle{amsalpha}

\end{document}